\documentclass[reqno]{amsart}
\usepackage{amsmath}
\usepackage{amssymb}

\usepackage{enumitem,kantlipsum}
\usepackage[margin=1in]{geometry}
\usepackage[all,cmtip]{xy}
\usepackage{color,graphicx}
\usepackage{amsthm}
\usepackage{comment}
\usepackage{mathrsfs}
\usepackage[toc,page]{appendix}
\usepackage[utf8]{inputenc}
\usepackage{tikz, subfigure}
\usepackage[mathscr]{euscript}

\usepackage[bookmarks=true, bookmarksopen=true,%
bookmarksdepth=3,bookmarksopenlevel=2,%
colorlinks=true,%
linkcolor=blue,%
citecolor=blue,%
filecolor=blue,%
menucolor=blue,%
urlcolor=blue]{hyperref}
\usepackage{cleveref}

\def\AA{\mathbb{A}}

\def\CC{\mathbb{C}}

\def\HH{\mathbb{H}}

\def\QQ{\mathbb{Q}}
\def\RR{\mathbb{R}}

\def\SS{\mathbb{S}}

\def\WW{\mathbb{W}}

\def\ZZ{\mathbb{Z}}

\def\calF{\mathcal{F}}

\def\calL{\mathcal{L}}

\def\calN{\mathcal{N}}

\def\calS{\mathcal{S}}

\def\bc{\mathbf{c}}

\newcommand\frD{\mathfrak{D}}

\newcommand\frH{\mathfrak{H}}

\newcommand\frS{\mathfrak{S}}

\newcommand\frg{\mathfrak{g}}

\newcommand\frl{\mathfrak{l}}

\newcommand\frs{\mathfrak{s}}
\newcommand\frt{\mathfrak{t}}

\newcommand\frz{\mathfrak{z}}

\newcommand{\Locsys}{\textup{Locsys}}

\newcommand{\Bun}{\textup{Bun}}

\newcommand{\Coh}{\textup{Coh}}

\newcommand{\Ind}{\textup{Ind}}

\newcommand\Loc{\textup{Loc}}

\newcommand{\QCoh}{\textup{QCoh}}

\newcommand\rank{\textup{rank}}

\newcommand{\Res}{\textup{Res}}

\newcommand\Tot{\textup{Tot}}

\newcommand{\Vect}{\textup{Vect}}

\newcommand{\MHC}{\textup{MHC}}

\newcommand\Hom{\textup{Hom}}
\newcommand\End{\textup{End}}
\newcommand\Map{\textup{Map}}

\newcommand\nc{\newcommand}
\nc\on{\operatorname}
\nc\ol{\overline}
\nc\ul{\underline}
\nc\us[1]{\underline{\smash{#1}}}

\nc\oo{\infty}

\nc\Cone{\mathit{Cone}}
\nc\ssupp{\mathit{ss}}
\nc\risom{\stackrel{\sim}{\to}}
\nc\Sh{\textup{Sh}}
\nc\un{\diamondsuit}
\nc\orient{\mathit{or}}
\nc\sing{\mathit{sing}}
\nc\MF{\on{MF}}
\nc\inthom{\mathit{Hom}}

\newcommand{\Ch}{\textup{Ch}}

\newcommand{\N}{\mathcal{N}}
\newcommand{\g}{\mathfrak{g}}

\newcommand{\colim}{\textup{colim}}

\newtheorem{thm}[equation]{Theorem}
\newtheorem{prop}[equation]{Proposition}
\newtheorem{lem}[equation]{Lemma}
\newtheorem{cor}[equation]{Corollary}

\theoremstyle{definition}
\newtheorem{defn}[equation]{Definition}

\newtheorem{defnthm}[equation]{Definition-Theorem}
\newtheorem{eg}[equation]{Example}

\newtheorem{rmk}[equation]{Remark}
\newtheorem{conj}[equation]{Conjecture}

\numberwithin{equation}{section}

\begin{document}

\nc{\Conn}{\mathrm{Conn}}

\nc{\fg}{\mathfrak g}
\nc{\fh}{\mathfrak h}

\nc{\cN}{\mathcal N}


\title[Derived categories of character sheaves]{Derived categories of character sheaves}

\author{Penghui Li}

\address{Institute of Science and Technology Austria}
\email{pli@ist.ac.at}

\begin{abstract}
We give a block decomposition of the dg category of character sheaves on a simple and simply-connected complex reductive group $G$, similar to the one in generalized Springer correspondence. As a corollary, we identify the category of character sheaves on $G$ as the category of quasi-coherent sheaves on an explicitly defined derived stack $\widehat{G}$.
\end{abstract}

\maketitle


\section{Introduction}

\subsection{Springer correspondence}
	Let $G$ be a connected complex reductive group, $\mathcal{N}_G$ be the unipotent variety in $G$. The category of equivariant sheaves on $\mathcal{N}_G$ has been the subject of many studies in geometric representation theory. The project is initiated by Springer \cite{Spr1,Spr2}, who constructed Springer resolution of $\cN_G$, which he used to give a geometric construction of the  Weyl group representations from its action on the cohomology of Springer fibers. Then Lusztig \cite{Lu7} extended the method and gave a description of all perverse sheaves on $\mathcal{N}_G$ in terms representations of various relative Weyl groups. And recently, Rider-Russell \cite{RR2} extended the results to derived setting. Now we have: 
\begin{thm}[Generalized Springer Correspondence \cite{Lu7,RR2}]
	\label{genespringer}
		Let $\Sh^{\heartsuit}_c(\cN_G/G)$ be the abelian category of $G$-equivariant perverse sheaves on $\cN_G$ and $D^b_c(\cN_G/G)$ be the bounded derived category of $G$-equivariant constructible sheaves on $\cN_G$.We have equivalences of abelian/triangulated categories:
		$$  \Sh^{\heartsuit}_c(\cN_G/G) \simeq \bigoplus_{\textbf{c} \in  \{ \textup{cuspidal data}\}/ \sim}  \CC[W^{L_\textbf{c}}]  \textup{ -mod}^{\heartsuit}_\textup{fd},  $$
		
		$$  D^b_c(\cN_G/G) \simeq \bigoplus_{\textbf{c} \in  \{ \textup{cuspidal data}\}/ \sim} D^b( \CC[W^{L_\textbf{c}}] \# \SS_{\frz^*_{L_\textbf{c}}[-2]} ),  $$
	where we denote	\begin{itemize}
			\item $L_\textbf{c}$ the Levi subgroup in the cuspidal data $\textbf{c}$,
			\item $W^{L_\textbf{c}}:= N_W(W_{L_\textbf{c}})/W_{L_\textbf{c}}$ the relative Weyl group of $(G,L_\textbf{c})$,
			\item $\frz_{L_\textbf{c}} :=$ the Lie algebra of  the center $Z_{L_\textbf{c}}$ of $L_\textbf{c}$,
			\item $\SS_V$:= the graded symmetric algebra on a (complex of) vector space $V$,
			\item  $\CC[\Gamma] \# A$ the smash product of $A$ with $\Gamma$, 
			\item  $A\textup{ -mod}^{\heartsuit}_\textup{fd}$ the abelian category of finite dimensional $A$-modules.
			
			\item $D^b(A)$ the bounded derived category of $A$-modules.
		
		\end{itemize}
\end{thm}

Choose $T \subset B \subset G$  a maximal torus and a Borel subgroup. Let $\Delta=\Delta(G,T,B)$ be the set of simple root. For $I \subset \Delta$, let $L_I$ be the subgroup of $G$ generated by $T, U_{\alpha}$ for $\alpha \in I $ or $-\alpha \in I$, where $U_\alpha$ is the one parameter subgroup of $\alpha$. 
For $L \subset K$ two Levi groups of $G$, put $W^L_K:=N_{W_K}(W_L)/W_L$, also set $W^L:=W^L_G$ and $W_K:=W^T_K$, this agree with the notation in above Theorem. For $I \subset I'$, set $W^I_{I'}:= W^{L_I}_{L_{I'}}$, $\frz_I:= Lie(Z_{L_I})$. Denote $C_G$ be the (finite) set of isomorphic class of cuspidal sheaves on $\mathcal{N}/G$, and $c_G:=|C_G|$. These numbers are explicitly calculated in \cite{Lu7}. Put $C_I = C_{L_I}$ , $c_I=|C_I|$. Also denote $\Sh(\N_G/G)$ the dg category of all $G$-equivariant sheaves on $\N_G$, and $\Sh_c(\N_G/G) \subset \Sh(\N_G/G)$ the dg category of constructible sheaves.  We prove the following upgrade of Theorem~\ref{genespringer}:
\begin{thm}
	\label{generalizedspringer}
	There are equivalence of abelian/dg categories:
	$$  \Sh^{\heartsuit}_c(\cN_G/G) \simeq \bigoplus_{I \subset \Delta}  (\CC[W^I] \textup{ -mod}^{\heartsuit}_\textup{fd})^{\oplus c_I}  $$
	$$  \Sh_c(\cN_G/G) \simeq \bigoplus_{I \subset \Delta}  (\CC[W^I] \# \SS_{\frz_I[1]} \textup{ -mod}_\textup{fd})^{\oplus c_I}  $$
	$$  \Sh(\cN_G/G) \simeq \bigoplus_{I \subset \Delta} (\CC[W^I] \# \SS_{\frz_I[1]}  \textup{ -mod} )^{\oplus c_I} $$
\end{thm}	

This improves the previous theorem in three aspects: (1) we pick an explicit set of representitives in the quotient $  \{ \textup{cuspidal data}\}/ \sim$. (2) we promote the statement of triangulated categories to dg categories,  the proof is an application of the general ``purity implies formality" result of Cirici-Hovey \cite{CH17}. This upgrade to dg category is necessary for our main theorem since triangulated category does not behave well under (co)limits. (3) we extends to statement to all sheaves via analysing compact objects.

\subsection{Character sheaves} Now we turn to the main object we study in this paper. Let $N \subset B$ unipotent radical, recall the horocycle correspondence: 
    $$\xymatrixrowsep{0pc}\xymatrix { Y:=(G/N \times  G/N )/T     &  G \times G/B \ar[l]^-{q} \ar[r]_-{p}  &       G    \\
    (g\tilde{x},   \tilde{x})     & \quad  (g,x)  \ar@{|->}[r] \ar@{|->}[l]  \quad &     g          }$$
    where $\tilde{x}$ is a lift of $x$ in $G/N$. We have the Radon transform 
   $R:=q_*p^!: \Sh(G/G) \rightarrow \Sh(Y/G)$ and its left adjoint $\check{R}:=p_!q^*: \Sh(Y/G) \rightarrow \Sh(G/G),$ between the corresponding $G$-equivariant dg-category of sheaves, where $G$ acts on $Y$ from the left and on $G$ by conjugation. Note that $T$ acts freely on $Y$ by acting on the first factor $G/N$ from the right.
\begin{defnthm}[\cite{MV,Gi1}]
	Let $\Ch(G)$ be the fully subcategory of $\Sh(G/G)$ consists of sheaf $F$ satisfying one of the following equivalent conditions:
	\begin{enumerate}
	    \item $F$ is a direct summand of $\check{R}(E)$, for some $E \in \Sh(Y/G)$ which is locally constant on $T$ orbits,
	    
	    \item $R(F)$ is locally constant on the $T$ orbits, 
	    
	    \item the singular support of $F$ is contained in $G \times \cN \subset G \times \fg \simeq T^*G$.

	\end{enumerate}	
Let $\Ch_c (G) \subset \Ch(G)$ be the full subcategory consists of constructible sheaves. And $\Ch^{\heartsuit}_c(G) \subset \Ch_c(G)$ be the abelian subcategory consists of perverse sheaves.
\end{defnthm}	
Assume that $G$ is simple and simply-connected. Denote by $\widetilde{\Delta}=\Delta \cup \{\alpha_0\}$ the set of affine simple roots. For $I \subsetneq \widetilde{\Delta}$, we extend the definition of $L_I, \frz_I$ by the same formula as in Theorem~\ref{generalizedspringer}, with the convention that $U_{\pm \alpha_0} :=U_{\mp \alpha_h}$, for $\alpha_h$ the highest root. Denote $\widetilde{W}$ the affine Weyl group, $\widetilde{W}_I$ the subgroup generated by reflections in $I$, and $\widetilde{W}^I := N_{\widetilde{W}}\widetilde{W}_I/\widetilde{W}_I$. $\widetilde{W}^I$ acts on the affine subspace $\cap_{\alpha \in I} \alpha^\perp \subset X_*(T) \otimes \RR$ (where $\alpha^\perp$ is the plane fixed by the simple reflection corresponding to $\alpha$), this induces linear action on $\frz_I$. We can state the main theorem of this paper:
\begin{thm}
\label{main}
There are equivalences between abelian/dg-categories:
$$\Ch^{\heartsuit}_c(G) \simeq \bigoplus_{I \subsetneq \widetilde{\Delta}}  (\CC[\widetilde{W}^I] \textup{ -mod}^{\heartsuit}_\textup{fd})^{\oplus c_I} $$
$$\Ch_c(G) \simeq \bigoplus_{I \subsetneq \widetilde{\Delta}}  (\CC[\widetilde{W}^I] \#  \SS_{\frz_I[1]} \textup{ -mod}_\textup{fd})^{\oplus c_I} $$
$$\Ch(G) \simeq \bigoplus_{I \subsetneq \widetilde{\Delta}} ( \CC[\widetilde{W}^I] \#   \SS_{\frz_I[1]} \textup{ -mod}) ^{\oplus c_I}$$	

\end{thm}

\begin{rmk} 	
	\label{cuspidal}
	Denote by $\Ch_c^{cusp}(G) \subset \Ch_c(G)$ the full subcategory of cuspidal sheaves (those vanish under all parabolic restrictions). It is well known \cite{BDS,Lu7} (which we learnt from Daniel Juteau) that 
	$\Ch_c^{cusp}(G) \simeq \bigoplus_{I \subset \widetilde{\Delta}, |I| = r :=\rank(G)} ( \CC[\widetilde{W}^I] \#   \SS_{\frz_I[1]} \textup{ -mod}_\textup{fd}) ^{\oplus c_I} \simeq \Vect_\textup{fd} ^{\oplus \Sigma_{|I|=r} c_I }$. The theorem extends this result to all character sheaves.
\end{rmk}

\begin{rmk}
	The adjoint quotient $G/G$ can be identified with $\Omega^1(S^1,\mathfrak{g})/C^{\infty}(S^1,G)$ using the gauge uniformization, by sending a connection to its monodromy. Kitchloo \cite{Kit} defined topological Tits building for Kac-Moody groups $K$ which specializes to $\Omega^1(S^1,\mathfrak{g})$ for $K$ the loop group of $G$. It is easy to see that for usual Tits building of loop group, the equivariant (constructible) functions/sheaves is naturally parametrized by the faces of fundamental alcove, which is naturally identified with the set $\{ I \subsetneq \widetilde{\Delta} \}$. So the decomposition in the Theorem is natural from this perspective.
\end{rmk}
Let us briefly mention some backgrounds as well as strategies for proving the theorem. The theory of character sheaves was introduced by Lusztig \cite{Lu1} via definition (1) above. It serves as a categorification of the character theory of finite group of Lie type, with the functor $\check{R}$ corresponds to Deligne-Lusztig induction, and $R$ to taking central characters. Later, Mirkovic-Vilonen \cite{MV} and Ginzburg \cite{Gi1} showed the equivalence of above definitions over $\CC$. Using definition \textup{(3)}, together with Nadler, we related $\Ch(G)$ with sheaves on unipotent varieties:
\begin{thm}[\cite{LN}]
	\label{glue}
	There are an equivalences of abelian/dg categories:
	$$\Ch^{(\heartsuit)}_{(c)}(G) \simeq \lim_{\{I \subsetneq \widetilde{\Delta} \}^{op}} \Sh^{(\heartsuit)}_{(c)}(\mathcal{N}_{L_I}/L_I)$$
	where in the limit, the inclusion $I \subset I'$ goes the parabolic restriction functor $\Res^{I'}_{I}:\Sh(\mathcal{N}_{L_{I'}}/L_{I'}) \rightarrow \Sh(\mathcal{N}_{L_I}/L_{I})$ along $P^{I'}_{I}:=<T, U_\alpha: \alpha \in I' \textup{ or } -\alpha \in I>.$ 
\end{thm}	
Combining this with Theorem~\ref{generalizedspringer}, we calculate  $\Ch(G)$ as described in our main theorem.

\subsection{Spectral description of charcter sheaves}
The normalizer of a parabolic subgroup of a Coxeter group $ N_{\widetilde{W}}\widetilde{W}_I$ and its quotient $\widetilde{W}^I$ has been studied in general, e.g in \cite{BH,Bor}. Denote by $\Lambda_I$ the subgroup of translation in $\widetilde{W}^I$, and $W^I:=\widetilde{W}^I/\Lambda_I$, this extends the previous definition of $W^I$ for $I \subset \Delta$. Let $\Lambda_I^*$ be the dual lattice of $\Lambda_I$, and $\check{S}_I:=\Lambda_I^* \otimes \CC^\times.$
When $c_I >0,$ $\widetilde{W}^I$ is a Coxeter group, hence $\widetilde{W}^I \simeq W^I \ltimes \Lambda_I$. Denote by $\mathcal{L}X:=X \times_{X \times X} X $ the derived loop space of $X$. We have the following spectral description of $\Ch(G):$
\begin{cor} 
	\label{spectral}
	Denote by $\widehat{G}:= \coprod_{I \subsetneq \widetilde{\Delta},c_I>0}(\mathcal{L}\check{S}_I / W^I)^{\coprod c_I}$ the derived stack and by $\widehat{G}_{cl}:= \coprod_{I \subsetneq \widetilde{\Delta},c_I>0}(\check{S}_I / W^I)^{\coprod c_I}$ the underline classical stack.  Then there is commutative diagram of equivalences of abelian/dg categories:
	$$\xymatrix{
	\Ch_c^{\heartsuit}(G) \ar[r]^-{\simeq} \ar[d] & \Coh_{0}^\heartsuit(\widehat{G}_{cl}) \ar[d] \\
	  \Ch_c(G) \ar[r]^-{\simeq} \ar@{^(->}[d] & \Coh_{0}(\widehat{G}) \ar@{^(->}[d]  \\
	\Ch(G) \ar[r]^-{\simeq} & \QCoh(\widehat{G})}$$ where $\Coh^{\heartsuit}_{0} /  \Coh_{0}$ denotes the abelian/dg category of coherent sheaves with $0$-dimensional support, and $\QCoh$ the dg category of quasi-coherent sheaves.
\end{cor}

We can parametrize of irreducible character sheaves, and calculated  their Hom complex  from the spectral side:

\begin{cor} 
	\label{irreducible character sheaves}
	Denote $\textup{Irr}(G)$ the set of irreducible character sheaves. There is a bijection of set  
		$$\textup{Irr}(G) \longleftrightarrow \{(I,F,s,\rho) | I \subsetneq \widetilde{\Delta}, F \in C_I, s \in \check{S}_I/W^I, \text{ and } \rho \text{ an irreducible representation of } W^I_s \},$$
		 such that
		\begin{enumerate}
			\item   $\Hom_{\Sh(G/G)}^*(\calF_{I,F,s,\rho}, \calF_{I',F',s',\rho'})=0, \text{ for any } (I,F,s) \neq (I',F',s')$,
			\item   $\Hom_{\Sh(G/G)}^*(\calF_{I,F,s,\rho}, \calF_{I,F,s,\rho'})= \Hom^*_{W^I_s} (\rho,\SS_{\frz_I^*[-1] \oplus \frz_I^*[-2]} \otimes \rho' ),$
		\end{enumerate}
	where  $\calF_{I,F,s,\rho} \in \textup{Irr}(G)$ denote the character sheaf corresponding to $(I,F,s,\rho)$. 
\end{cor}

\begin{rmk}
	\label{corollaryrmk}
	\begin{enumerate}[wide, labelwidth=!, labelindent=0pt]
\item 
As the notation suggested, $\widehat{G}$ is analogous to the unitary dual of $G$ in the representation theory of real groups.			 
\item 
Let $\widehat{G}^{0} \subset \widehat{G}$ consists of $0$-dimensional components.  
We see from Remark~\ref{cuspidal} that inside the above equivalence, we have $\Ch^{cusp}(G) \simeq \QCoh(\widehat{G}^{0})$. 
\item 
The analogous statement for a torus $T$ is:  take $\widehat{T} = \mathcal{L}\check{T}$, we have $\Ch(T) = \Loc(T/T) \simeq \Loc(T) \times \Loc(BT) \simeq \QCoh(\check{T}) \times \QCoh(\frt^*[-1]) \simeq \QCoh(\mathcal{L}\check{T})=\QCoh(\widehat{T})$, where $\Loc$ denotes the dg category of local systems.

\item 
By \cite{AHJR14}, the constant sheaf on $G$ corresponds to $(\emptyset, *, 1,\epsilon)$, where $* \in C_{T}$ is the unique element (which corresponds to the constant sheaf at $1 \in T$), $1 \in \check{T}/W$ the identify and $\epsilon$ is the sign representation of $W$. In this case, last Corollary recovers the well known result on the cohomology of adjoint quotient:
$H^*(G/G) = \Hom^*_{W}(\epsilon,\SS_{\frt^*[-1] \oplus \frt^*[-2]} \otimes \epsilon ) \simeq \Hom^*_{W}(\CC,\SS_{\frt^*[-1] \oplus \frt^*[-2]} \otimes \CC ) \simeq \SS_{\frt^*[-1] \oplus \frt^*[-2]}^{W}.$


\end{enumerate}
\end{rmk}

\begin{eg}
	\label{example}
We give examples of $\widehat{G}$ for $G$ of rank $\leq 3$. Recall that the set $\{ I \subsetneq \widetilde{\Delta} \}$ can be identifed with the set of faces of the affine alcove. Below we will draw the affine simple coroots and affine alcove in $X_*(T) \otimes \RR$. We denote:
\begin{itemize}
	\item[]  
	 \begin{tikzpicture} 
	\draw node[circle,fill,inner sep=1pt]{} ;
	\end{tikzpicture} : a vertex of the affine alcove with $c_I \neq 0$;
	\item[] 
	\begin{tikzpicture}[
	circ/.style={
		circle,
		fill=white,
		draw,
		outer sep=0pt,
		inner sep=1pt
	}]
  		\draw node[circ]{};
    \end{tikzpicture}
    : a vertex of the affine alcove with $c_I = 0$;
    \item[]
    \begin{tikzpicture} 
    \draw [very thick] (0,0) -- (1,0);
    \end{tikzpicture} : an edge of the affine alcove with $c_I \neq 0$;
       \item[]
    \begin{tikzpicture} 
    \draw [very thick, dashed] (0,0) -- (1,0);
    \end{tikzpicture} : an edge of the affine alcove with $c_I = 0$;
           \item[]
    \begin{tikzpicture} 
   \draw [fill=lightgray,lightgray] (0,0) --(1,0) --(1,0.3) -- (0,0) ;
    \end{tikzpicture} : a $2$-dimensional face of the affine alcove with $c_I \neq 0$;
      \item[]
    \begin{tikzpicture} 
    \draw [fill=white] (0,0) --(1,0) --(1,0.3) -- (0,0) ;
    \end{tikzpicture} : a $2$-dimensional face of the affine alcove with $c_I = 0$;
    \item[]
   \begin{tikzpicture} 
   \draw [->] (0,0) --(1,0) ;
   \end{tikzpicture} : a simple coroot or the highest coroot $\alpha^\vee_h$; 
    \item[]
   numbers $1,2,...: c_I$ of the corresponding face $I$;
    \item[]
   $\mathfrak{S}_n$: the symmetric group on $n$ letters;
   \item[]
   $\mathfrak{D}_n$: the dihedral group of order $2n$.
\end{itemize}	
\begin{itemize}
\item	$G=A_1$
	
$$ 
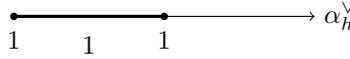
\begin{tikzpicture}
\draw [very thick] (0,0) node[circle,fill,inner sep=1pt,label=below:$1$](a){} -- (1,0) node[label=below:$1$](c){} -- (2,0) node[circle,fill,inner sep=1pt,label=below:$1$](b){};
\draw [->] (0,0) -- (4,0) node[right]{$\alpha^\vee_h$};
\end{tikzpicture}$$
$\widehat{G}= \mathcal{L}\CC^\times /\frS_2 \coprod * \coprod *.$

\item $G=A_2$
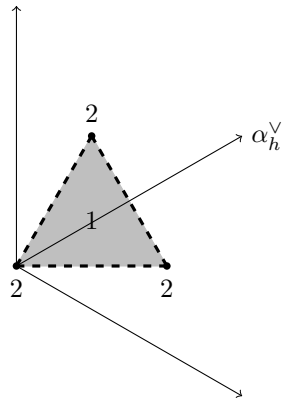
$$ \begin{tikzpicture}
\draw [fill=lightgray,lightgray] (0,0) --(2,0) --(1,1.73) -- (0,0) ;
\draw [->] (0,0) -- (3,-1.73) ;
\draw [->] (0,0) -- (3,1.73) node[right]{$\alpha^\vee_h$} ;
\draw [->] (0,0) -- (0,3.46) ;
\draw [dashed,very thick] (0,0)  -- (2,0) ;
\draw [dashed,very thick] (2,0) -- (1,1.73);
\draw [dashed,very thick] (0,0) -- (1,1.73);
\node at (1,0.6) {1};
\draw (0,0) node[circle,fill,inner sep=1pt,label=below:$2$]{};
\draw (2,0) node[circle,fill,inner sep=1pt,label=below:$2$]{};
\draw (1,1.73) node[circle,fill,inner sep=1pt,label=above:$2$]{};
\end{tikzpicture}$$
$\widehat{G}= \mathcal{L}(\CC^{\times})^2/\frS_3 \coprod *^{\coprod 2} \coprod *^{\coprod 2} \coprod *^{\coprod 2}.$

\item $G=B_2$

$$ 
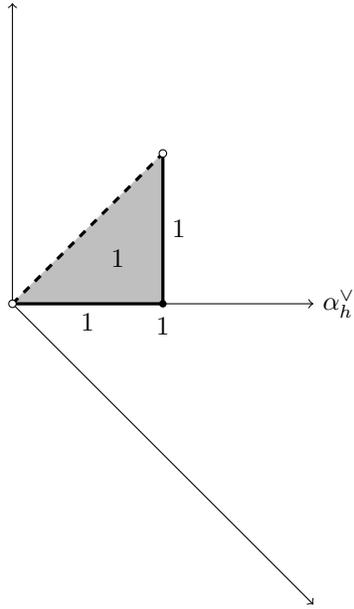
\begin{tikzpicture}[
circ/.style={
	circle,
	fill=white,
	draw,
	outer sep=0pt,
	inner sep=1pt
}]
\draw [fill=lightgray,lightgray] (0,0) --(2,0) --(2,2) -- (0,0) ;
\node [below] at (1,0) {1};
\node [right] at (2,1) {1};
\node at (1.4,0.6) {1};
\draw [->] (0,0) -- (4,0) node[right]{$\alpha^\vee_h$} ;
\draw [->] (0,0) -- (0,4) ;
\draw [->] (0,0) -- (4,-4) ;
\draw [dashed, very thick] (0,0)  -- (2,2) ;
\draw [very thick] (0,0) -- (2,0);
\draw [very thick] (2,0) --(2,2);
\draw  (0,0) node[circ] {} ;
\draw  (2,2) node[circ] {} ;
\draw (2,0) node[circle,fill,inner sep=1pt,label=below:$1$](b){};
\end{tikzpicture}$$
$\widehat{G}=\mathcal{L}(\CC^{\times})^2/\frD_4 \coprod \calL\CC^{\times}/\frS_2 \coprod \calL\CC^{\times}/\frS_2 \coprod *$.

\item $G=G_2$
$$ 
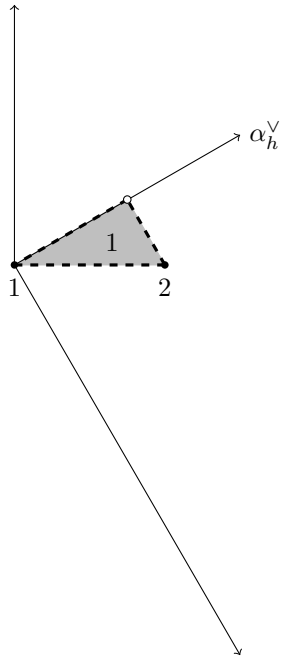
\begin{tikzpicture}[
circ/.style={
	circle,
	fill=white,
	draw,
	outer sep=0pt,
	inner sep=1pt
}]
\draw [fill=lightgray,lightgray] (0,0) --(2,0) --(1.5,0.87) -- (0,0) ;
\node at (1.3,0.3) {1};
\draw [dashed, very thick] (0,0) --(2,0);
\draw [dashed, very thick] (2,0) --(1.5,0.87);
\draw [dashed, very thick] (1.5,0.87) -- (0,0);

\draw [->] (0,0) -- (0,3.46);
\draw [->] (0,0) -- (3,-5.19);
\draw [->] (0,0) --(3,1.73) node[right]{$\alpha^\vee_h$};

\draw  (1.5,0.87) node[circ] {} ;
\draw (2,0) node[circle,fill,inner sep=1pt,label=below:$2$]{};
\draw (0,0) node[circle,fill,inner sep=1pt,label=below:$1$]{};
\end{tikzpicture}$$
$\widehat{G}= \calL(\CC^\times)^2/\frD_6 \coprod * \coprod *^{\coprod 2}.$
\end{itemize}
For the rank $3$ examples below, we omit the labeling $c_I$ for the interior of affine alcove (which is always $1$). We also draw 
\begin{tikzpicture}[scale=0.2, rotate around x = 10, rotate around y =-5, rotate around z = -5]
\coordinate (A1) at (0,0,0);
\coordinate (A2) at (0,2,0);
\coordinate (A3) at (1,1,1);
\coordinate (A4) at (1,1,-1);

\coordinate (O) at (-1,-1,-1);
\coordinate (A) at (-1,1,-1);
\coordinate (B) at (-1,1,1);
\coordinate (C) at (-1,-1,1);
\coordinate (D) at (1,-1,-1);
\coordinate (E) at (1,1,-1);
\coordinate (F) at (1,1,1);
\coordinate (G) at (1,-1,1);

\draw [dashed] (O) -- (C);
\draw (C) -- (G);
\draw (G) -- (D);
\draw [dashed] (D) -- (O);
\draw [dashed] (O) -- (A);
\draw (A) -- (E);
\draw (E) -- (D);
\draw (A) -- (B);
\draw (B) -- (C);
\draw (E) -- (F);
\draw (F) -- (G);
\draw (B) -- (F);
\end{tikzpicture}
 a reference cube centered at origin.
\begin{itemize}
\item $G=A_3$
$$
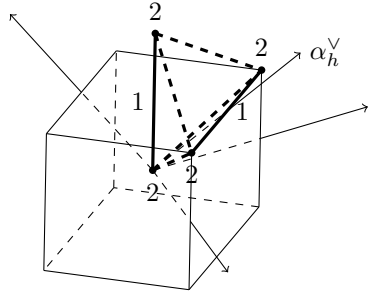
\begin{tikzpicture}[rotate around x = 10, rotate around y =-5, rotate around z = -5]
\coordinate (A1) at (0,0,0);
\coordinate (A2) at (0,2,0);
\coordinate (A3) at (1,1,1);
\coordinate (A4) at (1,1,-1);

\coordinate (O) at (-1,-1,-1);
\coordinate (A) at (-1,1,-1);
\coordinate (B) at (-1,1,1);
\coordinate (C) at (-1,-1,1);
\coordinate (D) at (1,-1,-1);
\coordinate (E) at (1,1,-1);
\coordinate (F) at (1,1,1);
\coordinate (G) at (1,-1,1);

\draw [dashed] (O) -- (C);
\draw (C) -- (G);
\draw (G) -- (D);
\draw [dashed] (D) -- (O);
\draw [dashed] (O) -- (A);
\draw (A) -- (E);
\draw (E) -- (D);
\draw (A) -- (B);
\draw (B) -- (C);
\draw (E) -- (F);
\draw (F) -- (G);
\draw (B) -- (F);
\draw [dashed] (A1) -- (1,0,1);
\draw [->] (1,0,1) -- (2,0,2);
\draw [dashed] (A1) -- (1,0,-1);
\draw [->] (1,0,-1) -- (2,0,-2);
\draw [dashed] (A1) -- (-1,1,0);
\draw [->] (-1,1,0) -- (-2,2,0);
\draw [dashed] (A1) -- (1,1,0);
\draw [->] (1,1,0) -- (2,2,0) node[right]{$\alpha^\vee_h$} ;

\draw [very thick] (A1) -- (A2);
\draw [very thick] (A3) -- (A4);
\draw [dashed, very thick] (A1) -- (A3);
\draw [dashed, very thick] (A1) -- (A4);
\draw [dashed, very thick] (A2) -- (A4);
\draw [dashed, very thick] (A2) -- (A3);
\draw (0,0,0) node[circle,fill,inner sep=1pt,label=below:$2$]{};
\draw (0,2,0) node[circle,fill,inner sep=1pt,label=above:$2$]{};
\draw (1,1,1) node[circle,fill,inner sep=1pt,label=below:$2$]{};
\draw (1,1,-1) node[circle,fill,inner sep=1pt,label=above:$2$]{};
\node [left] at (0,1,0) {1};
\node [right] at (1,1,0) {1};
\end{tikzpicture}$$
$\widehat{G}=\calL(\CC^{\times})^3/\frS_4 \coprod \calL\CC^{\times}/\frS_2 \coprod \calL\CC^{\times}/\frS_2 \coprod *^{\coprod 2} \coprod *^{\coprod 2} \coprod *^{\coprod 2} \coprod *^{\coprod 2}.$

\item $G=B_3$
$$
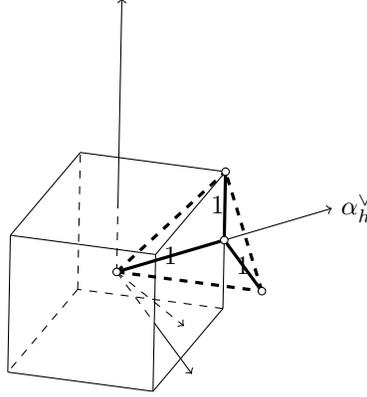
\begin{tikzpicture}[rotate around x = 10, rotate around y = -5, rotate around z = -5,
circ/.style={
	circle,
	fill=white,
	draw,
	outer sep=0pt,
	inner sep=1pt
}]

\coordinate (O) at (-1,-1,-1);
\coordinate (A) at (-1,1,-1);
\coordinate (B) at (-1,1,1);
\coordinate (C) at (-1,-1,1);
\coordinate (D) at (1,-1,-1);
\coordinate (E) at (1,1,-1);
\coordinate (F) at (1,1,1);
\coordinate (G) at (1,-1,1);

\draw [dashed] (O) -- (C);
\draw (C) -- (G);
\draw (G) -- (D);
\draw [dashed] (D) -- (O);
\draw [dashed] (O) -- (A);
\draw (A) -- (E);
\draw (E) -- (D);
\draw (A) -- (B);
\draw (B) -- (C);
\draw (E) -- (F);
\draw (F) -- (G);
\draw (B) -- (F);

\coordinate (U) at (0,0,0);
\coordinate (V) at (2,0,0);
\coordinate (W) at (1,0,-1);
\coordinate (X) at (1,1,-1);

\draw [dashed] (U) -- (0,1,0);
\draw [->] (0,1,0) -- (0,4,0);
\draw [dashed] (U) -- (1,0,1);
\draw [->] (1,0,1) -- (2,0,2);
\draw [->] (U) -- (2,0,-2) node[right]{$\alpha^\vee_h$} ;
\draw [->, dashed] (U) -- (0,-2,-2);

\draw [dashed, very thick] (U) -- (V);
\draw [very thick] (W) -- (X);
\draw [very thick] (U) -- (W);
\draw [dashed, very thick] (U) -- (X);
\draw [dashed, very thick] (V) -- (X);
\draw [very thick] (V) -- (W);
\draw (U) node[circ]{};
\draw (V) node[circ]{};
\draw (W) node[circ]{};
\draw (X) node[circ]{};
\node at (0.5,0,-0.5) {1};
\node at (1.5,0,-0.5) {1};
\node at (0.9,0.5,-1) {1};
\end{tikzpicture}$$
$\widehat{G}= \calL(\CC^{\times})^3/ (\frS_3 \ltimes \{\pm 1\}^3) \coprod \calL \CC^{\times}/\frS_2 \coprod \calL \CC^{\times}/\frS_2 \coprod \calL \CC^{\times}/\frS_2.$

\item $G=C_3$
$$
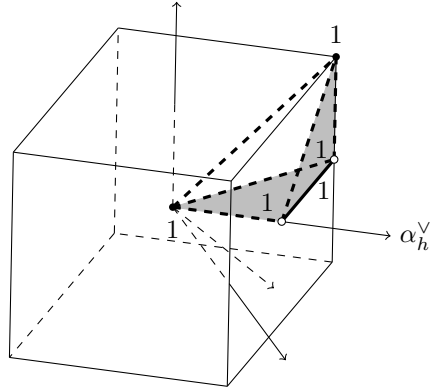
\begin{tikzpicture}[scale=1.5, rotate around x = 10, rotate around y = -5, rotate around z = -5,
circ/.style={
	circle,
	fill=white,
	draw,
	outer sep=0pt,
	inner sep=1pt
}]

\coordinate (U) at (0,0,0);
\coordinate (V) at (1,0,0);
\coordinate (W) at (1,0,-1);
\coordinate (X) at (1,1,-1);
\draw [fill=lightgray,lightgray] (U) --(V) --(W) -- (U);
\draw [fill=lightgray,lightgray] (X) --(V) --(W) -- (X);

\coordinate (O) at (-1,-1,-1);
\coordinate (A) at (-1,1,-1);
\coordinate (B) at (-1,1,1);
\coordinate (C) at (-1,-1,1);
\coordinate (D) at (1,-1,-1);
\coordinate (E) at (1,1,-1);
\coordinate (F) at (1,1,1);
\coordinate (G) at (1,-1,1);

\draw [dashed] (O) -- (C);
\draw (C) -- (G);
\draw (G) -- (D);
\draw [dashed] (D) -- (O);
\draw [dashed] (O) -- (A);
\draw (A) -- (E);
\draw (E) -- (D);
\draw (A) -- (B);
\draw (B) -- (C);
\draw (E) -- (F);
\draw (F) -- (G);
\draw (B) -- (F);

\draw [dashed] (U) -- (0,1,0);
\draw [->] (0,1,0) -- (0,2,0);
\draw [dashed] (U) -- (1,0,1);
\draw [->] (1,0,1) -- (2,0,2);
\draw [dashed] (U) -- (1,0,0);
\draw [->] (1,0,0) -- (2,0,0) node[right]{$\alpha^\vee_h$} ;
\draw [->, dashed] (U) -- (0,-2,-2);

\draw [dashed, very thick] (U) -- (V);
\draw [dashed, very thick] (W) -- (X);
\draw [dashed, very thick] (U) -- (W);
\draw [dashed, very thick] (U) -- (X);
\draw [dashed, very thick] (V) -- (X);
\draw [very thick] (V) -- (W);

\draw (U) node[circle,fill,inner sep=1pt,label=below:$1$]{};
\draw (V) node[circ]{};
\draw (W) node[circ]{};
\draw (X) node[circle,fill,inner sep=1pt,label=above:$1$]{};
\node [right] at (1,0,-0.5) {1};
\node at (0.75,0,-0.25) {1};
\node at (1,0.25,-0.75) {1};

\end{tikzpicture}$$
$\widehat{G} = \calL(\CC^{\times})^3/ (\frS_3 \ltimes \{\pm 1\}^3) \coprod \calL (\CC^\times)^2/\frD_4 \coprod \calL(\CC^\times)^2/\frD_4 \coprod \calL \CC^{\times}/\frS_2 \coprod * \coprod *.$

\end{itemize}

\end{eg}

\subsection{Relation to Betti Geometric Langlands Conjecture for nodal curve}

Let $\Sigma$ be a compact Riemann surface, $\Bun_G(\Sigma)$ be the moduli stack of principal $G$-bundles on $\Sigma$. Denote $\check{G}$ the Langlands dual group of $G$, and $\Locsys_{\check{G}}(\Sigma)$ the derived moduli stack of (Betti) $\check{G}$-local systems on $\Sigma$. Ben-Zvi-Nadler has proposed:

\begin{conj}[\cite{BZN16}] There is an equivalence of dg-categories 
$$	\Sh_\N (\Bun_G(\Sigma)) \simeq  \textup{IndCoh}_{\check{\N}} (\Locsys_{\check{G}}(\Sigma))$$
between the dg-category of sheaves with nilpotent singular support on $\Bun_G(\Sigma)$ and the dg-category of Ind-coherent sheaves with nilpotent singular support.
\end{conj}	

Denote $\Bun^{0,ss}_G(\Sigma) \subset \Bun_G(\Sigma)$ the open substack consists of degree 0 and semistable bundles. Take $\Sigma=C_0$ the nodal curve, we have $\Bun^{0,ss}_G(C_0) \simeq G/G$ and $\Sh_\N(\Bun^{0,ss}_G(C_0)) \simeq \Ch(G)$ . Hence we expect that Corollary~\ref{spectral} appears as ``semistale part'' of above Conjecture for $C_0$. By normalization, we have $C_0 \simeq S^2 \cup_{* \coprod *} * ,$ and $\Locsys_{\check G}(C_0) \simeq \check{\frg}[-1]/ \check G \times_{B \check G \times B\check G} B\check G \simeq (\check{\frg}[-1] \times  \check G)/ \check G$. From this perspective, the derived structure on $\widehat{G}$ has a natural explanation: it should come from the 2-cell in $C_0$. 

In the case for torus $T$, Conjecture can be easily verified: $\Sh_\N(\Bun_T(C_0)) \simeq \Loc(X_*(T) \times T/T ) \simeq \Loc(X_*(T)) \otimes \Loc(T/T) \simeq \QCoh(B\check{T}) \otimes \QCoh(\widehat{T}) \simeq \QCoh(\Locsys_{\check T}(C_0)) \simeq \textup{IndCoh}_{\check{\N}}(\Locsys_{\check T}(C_0)).$ We see that the equivalence in Remark~\ref{corollaryrmk} (3) appears as the degree 0 semistable part corresponding to $0 \in X_*(T)$.

\section{Preliminary on dg categories}

We shall work in the context of pre-triangulated dg categories, or equivalently $k$-linear stable $\infty$-categories, for $char(k)=0$. See \cite{Toe07, Lur2} for reference. Denote $\Vect_{\textup{(fd)}}$ the dg category of (finite dimensional) vector spaces,  $A\textup{-mod}_\textup{fd}$ the dg category of (finite dimensional) $A$-modules and $A\textup{-perf}$ the dg category of perfect $A$-modules.
\subsection{Karoubi and compact generation}
Let $\mathscr{C}_0$ be an (non-cocomplete) idempotent complete dg category, we say a collection of objects $S$ \textit{Karoubi generate} $\mathscr{C}_0$ if every object in $\mathscr{C}_0$ can be obtained from $S$ by finite interation of cones of a morphism, and taking direct summand of a object. 

\begin{prop}
	\label{Karoubigenerate}
	If $\mathscr{C}_0$ is Karoubi generate by an object $c$, then $\mathscr{C}_0 \simeq \End(c)$-\textup{perf}.
\end{prop}

let $\mathscr{C}$ be cocomplete dg category, an object $c \in \mathscr{C}$ is \textit{compact} if the functor $\Map(c,-): \mathscr{C} \to \Vect$ preserves colimit. A collection of objects $S$ $\textit{generate}$ $\mathscr{C}$ if for any $c \in C$, we have $\Map(s,c)=0, \forall s \in S \Rightarrow c =0.$ Denote $\mathscr{C}^c \subset \mathscr{C}$ the full subcategory consists of compact objects.

\begin{prop}
	\label{compactgenerate}
	If a cocomplete dg category $\mathscr{C}$ is generated by a compact object $c$, then $\mathscr{C} \simeq \End(c)$-\textup{mod}.
\end{prop}

\begin{prop}
	\label{preserve compact object}
Let $F:\mathscr{C}_1 \to \mathscr{C}_2$ be a continuous functor between two dg categories. Assume that $F$ has a continuous right adjoint, then $F(\mathscr{C}^c_1) \subset \mathscr{C}^c_2$.
\end{prop}	

\subsection{Group actions on dg categories}
\label{group action on category}
Let $\Gamma$ be a discrete group acting (strictly) on a dg algebra $A$,  the smash product $\CC[\Gamma] \# A$ is by definition a dg algebra whose underline dg vector space is $\CC[\Gamma] \otimes A$, and the multiplication is given by $(\gamma \otimes a) \cdot (\gamma' \otimes a'):=(\gamma\gamma' \otimes \gamma'(a)a')$. On the other hand, $\Gamma$ acts on the category $A\textup{-mod}$, which we view as a functor $Act: B\Gamma \to \textup{DGCat}:=$ the category of dg categories. Denote by $Act^L: B\Gamma^{op} \to \textup{DGCat}_\textup{cont}:=\textup{the category of dg categories with colimit preserving functors}$, where  $Act^L(\gamma):=$ the left adjoint of $Act(\gamma)$, for $\gamma$ a morphism of $B\Gamma$. Denote $A\textup{-mod}^{\Gamma}:=\lim Act$, and $A\textup{-mod}_{\Gamma}:=\colim Act^L$ the invariants and coinvariants. The following equivalent categories are by definition the dg category of $\Gamma$-equivariant  $A$-modules:
\begin{prop} There are equivariances of dg categories:
 $$ \CC[\Gamma] \# A\textup{-mod} \simeq A\textup{-mod}^{\Gamma} \simeq A\textup{-mod}_{\Gamma}$$
\end{prop}

\begin{proof}
	The first equivalence is by applying Barr-Beck theorem the induction/restriction adjoint functors   $ ind:\CC[\Gamma] \# A\textup{-mod} \rightleftarrows A\textup{-mod}: res$. For the second equivalence, one can calculate limit in $\textup{DGCat}$ from colimit in $ \textup{DGCat}_\textup{cont}$ by passing to left adjoints  \cite[Lemma 1.3.3]{Gai}.
\end{proof}

\subsection{Local systems and Koszul duality}

Let $X$ be a topological space, denote $\textup{Open}(X)$ the category of open subset of $X$. A \textit{sheaf} $F$ of (complexes of) vector space is by definition a functor $F: \textup{Open}(X) \to \Vect,$ from the category of open subset of $X$ to the dg category of vector spaces, such that for any $\{U_i\}_{i \in I}$ open cover of $X$, the natural map $F(X) \to \Tot(F(U^\bullet_X))$ is an (quasi-)isomorphism, where $F(U^n_X) := \prod_{(i_1,i_2,...,i_n) \in I^n} F(U_{i_1} \cap U_{i_2} \cap ... \cap U_{i_n} )$ form a cosimplicial object in \Vect, and $\Tot(\cdot)$ denote its totalization. All sheaves on $X$ forms a dg category, denote by $\Sh(X)$, it is complete and cocomplete. 

\begin{eg}
	The singular cochains $C^*_X: \textup{Open}(X) \to \Vect$, via $U \mapsto C^*(U,\CC)$ is a sheaf on $X$. It is isomorphic to the sheafification of constant presheaf $U \mapsto \CC$.
\end{eg}	
A \textit{local system} $L$ on $X$ is a sheaf, such that there is an open cover of $\{U_i\}_{i \in I}$ of $X$, satisfying $\forall i \in I$, $L|_{U_i} \simeq C^*_{U_i} \otimes V_i$, for some $V_i \in \Vect$. Denote $\Loc(X)$ the dg category of local system on $X$. Assume $X$ is path connected, denote $\Omega X$ the based loops space of $X$. Its chains $C_{-*}(\Omega X)$ is naturally a dg algebra via the compostition of loops $\Omega X \times \Omega X \to \Omega X$. Denote also $\Pi X$ the fundamental $\infty$-groupoid of $X$.  We have equivalences of dg categories:
\begin{equation}
\label{local systems as modules}
\Loc(X) \simeq \Map(\Pi X, \Vect) \simeq C_{-*}(\Omega X) \textup{-mod}
\end{equation}
These equivalence are compatible with the natural $t$-structure. Taking hearts, we get the more familar equivalence of abelian categories: $$\Loc^{\heartsuit}(X) \simeq \pi_1(X)\textup{-mod}^{\heartsuit} \simeq \CC[\pi_1(X)] \textup{-mod}^{\heartsuit}.$$

Denote $\Loc_c(X) \subset \Loc(X)$ the full subcategory consist of finite rank local systems (i.e those with $V_i \in \Vect_{\textup{fd}}$).

\begin{thm}[Koszul duality for topological spaces]
	Let $X$ be a path connected, simply connected topological space. There is equivalence of dg categories:
	$$C_{-*}(\Omega X) \textup{-mod}_\textup{fd} \simeq \Loc_c(X) \simeq C^*({X}) \textup{-perf}$$
\end{thm}	
\begin{proof}
	The first equivalence follows from (\ref{local systems as modules}) above. For the second, $\Loc_c(X)$ is Karoubi generated by $C^*_X$, hence by Propsition~\ref{Karoubigenerate}, we have $\Loc_c(X) \simeq \End(C^*_X)\textup{-perf} = C^*(X)\textup{-perf}$ .
\end{proof}	

\begin{cor}
	\label{Koszul duality}
Let $V$ be a vector space, $W$ a finite group acts linearly on $V$. There is equivalence of dg categories:
$$\CC[W]\# \SS_{V[1]} \textup{-mod}_\textup{fd} \simeq \CC[W]\# \SS_{V[-2]} \textup{-perf}$$
\end{cor}	
\begin{proof}
	Choose a torus $T$ and an isomorphism $Lie(T) \simeq V$, we have  $C_{-*}(\Omega BT) = C_{-*}(T) \simeq \SS_{V[1]}$ and $C^*(BT) \simeq \SS_{V^*[-2]}$. Apply Koszul duality to $X=BT$, we get $\SS_{V[1]}\textup{-mod}_\textup{fd} \simeq C_{-*}(\Omega BT)\textup{-mod}_\textup{fd} \simeq C^*(BT)\textup{-perf} \simeq \SS_{V^*[-2]}\textup{-perf}.$ We conclude by taking $W$ invariants on both sides. Note that since $W$ is finite, a $\CC[W]\# \SS_{V^*[-2]}$ module is perfect if and only if its restriction as $\SS_{V^*[-2]}$ module is perfect.
\end{proof}

\subsection{Compact generations of weakly constructible sheaves}

A \textit{stratified space} $(X,\mathcal{S})$ is a pair consist of a topological space $X$ and a topological stratification $\mathcal{S}$ of $X$. A sheaf $F$ on $X$ is \textit{weakly $\mathcal{S}$-constuctible} if for any $S \in \mathcal{S}$ a stratum, the sheaf $F|_S$ is a local system on $S$. A sheaf $F$ is \textit{$\mathcal{S}$-constructible} if $F|_S$ is a finite rank local system. Denote by $\Sh_\mathcal{S}(X)$ (resp. $\Sh_{\mathcal{S},c}(X)$) the dg category of weakly $\mathcal{S}$-constructible (resp. $\mathcal{S}$-constructible) sheaves on $X$. $\Sh_\mathcal{S}(X)$ is complete and cocomplete. 


\begin{lem}
	\label{adjoints exist}
Denote $j:S  \hookrightarrow X$ the inclusion of a stratum. Then the functor $j^!: \Sh_{\mathcal{S}}(X) \to \Loc(S)$ is continuous.
\end{lem}

\begin{proof}
	Let $ F_\alpha \in \Sh_{\mathcal{S}}(X)$, with $\colim F_\alpha = F$. We need to show that the natural map $\eta:  \colim j^!(F_\alpha)  \to j^!(F) $ is an isomorphism in $\Loc(S)$. Suffices to check this locally on $S$. For any $s \in S$, 	there is an open neighborhood $U$ of $s$ in $X$, a stratified space $(Y, \mathcal{T})$, and an isomorphism $(U, \mathcal{S}|_U) \simeq (U \cap S) \times  C(Y, \mathcal{T})$ (where $C(Y, \mathcal{T})$ denotes the open cone of $(Y, \mathcal{T})$),  such that the diagram of stratified space commutes:
	$$\xymatrix{      &       (U, \mathcal{S}|_U)  \ar[d]^\simeq \\
		(U \cap S)   \ar[ur]^j   \ar[r]^-{j_0}   &   (U \cap S) \times  C(Y, \mathcal{T})  }                      $$
	
	$j_0^!$ is isomorphic to $p_!$, where $p: (U \cap S) \times  C(Y, \mathcal{T}) \to (U \cap S)$ is the projection. Now $p_!$ admit a right adjoint $p^!$, hence $j_0^!$ preserve colimits. This implies $\eta|_U$ is an isomorphism, and hence $\eta $ is an isomorphism.
\end{proof}

\begin{prop}
	The cocomplete dg category $\Sh_{\mathcal{S}}(X)$ is compactly generated.
\end{prop}	
\begin{proof}
	For each $S \in \calS$, we have $\Loc(S) \simeq C_{-*}(\Omega S)$-mod. Let $O_S \in \Loc(S)$ be the local system corresponding to the free module $C_{-*}(\Omega S)$. Then $O_S$ is a compact object in $\Loc(S)$. By Lemma~\ref{adjoints exist}, $j_!$ has a continuous right adjoint $j^!$, therefore by Proposition~\ref{preserve compact object}, $j_!(O_S)$ is a compact object in $\Sh_{\mathcal{S}}(X)$. Now the objects $\{ j_!(O_S)\}_{S \in \calS}$ generated $\Sh_{\mathcal{S}}(X)$.
\end{proof}

\begin{rmk}
In fact, \cite{Nad16, GPS18} proved that for any Lagrangian subvariety $\Lambda$ of $T^*X$, the category $\Sh_\Lambda(X)$ of sheaves with singular support in $\Lambda$  is compactly generated. The previous propsition corresponds to the case when $\Lambda = \coprod_{S \in \mathcal{S}} T^*_S X.$
\end{rmk}

Let $G$ be a topological group acting on $X$, define $\Sh(X/G):= \Tot(\Sh(G^{\times \bullet} \times X))$ the dg category of $G$-equivairant sheaves on $X$. 
Assume $G$ preserve a topological stratification $\mathcal{S}$, 
put $\Sh_{\mathcal{S},(c)}(X/G):= \Tot(\Sh_{G^{\times \bullet} \times \mathcal{S},(c)}(G^{\times \bullet} \times X))$. In the case $G$ has finite many orbits on $X$, the set of orbits $\mathcal{O}$ forms a topological stratification, we have $\Sh_{\mathcal{O}}(X/G) = \Sh(X/G)$, put $\Sh_c(X/G):= \Sh_{\mathcal{O},c}(X/G)$.

\begin{cor}
	\label{finite orbit stack}
	Let $G$ be an algebraic group acts on an algebraic variety $X$ with finitely many orbits, then $\Sh(X/G)$ is compactly generated, and $\Sh(X/G)^{c} \subset \Sh_c(X/G)$.
\end{cor}
\begin{proof} The first statement follows similar argument as in last Propsition. 
	Let $S$ be an orbit, $s \in S$, $G_s$ the stablizer at $s$. Let $O_{S/G} \in \Loc(S/G) \simeq \Loc(BG_s) \simeq C_{-*}(G_s)\text{-mod}$ be the object corresponding to the module $C_{-*}(G_s)$, then the objects $\{j_! O_{S/G}\}$ compactly generate $\Sh(X/G)$. Since $C_{-*}(G_s)$ is finite dimensional, $O_{S/G} \in \Sh_c(S/G)$ and therefore  $j_! O_{S/G} \in \Sh_c(X/G)$. 
\end{proof}

\section{Proof of Theorem~\ref{generalizedspringer}}

\begin{lem}
	Let $L$ be a connected reductive group such that there exist a cuspidal sheaf on $\mathcal{N}_L$. Assume that $L \subset P_i \subset G$, for some $P_i$ parabolic subgroups of $G$ with Levi $L$. Then all $P_i$'s are conjugate in $G$.
\end{lem}
\begin{proof}
	Let $T \subset L$ be a maximal torus, denote by $V_L \subset X_*(T) \otimes \RR$ be the subspace where roots of $L$ vanishes. Then a choice of such parabolic subgroup $P_i$ is a choice of chamber in $V_L \backslash \cup_{\alpha \in \Phi(G) \backslash \Phi(L)} \alpha^\perp$, where $\Phi(G)$ the set of roots of $G$ (The identification is given by chamber $C \mapsto < T,U_\alpha \; | \; \alpha(C) \geq 0>$).   Now by \cite[Theorem 9.2]{Lu7}, $N_W(W_L)/W_L$ is again Coxeter group acting on $V_L$ with hyperplanes $V_L \cap \alpha^\perp$, for $\alpha \in \Phi(G) \backslash \Phi(L)$. Hence $N_W(W_L)/W_L$ act transitivly on the set of chambers. Hence $N_W(W_L)$ acts transitively on the set of parabolic subgroups satisfying the condition required.
\end{proof}

Let $P \subset G$ be a parabolic subgroup with Levi factor $L$, we have the correspondence:
$$\xymatrix{\calN_L/L &  \calN_P/P \ar[l]_{q} \ar[r]^p & \N_G/G },$$
Denote the parabolic induction functor by $\Ind_P:= p_! q^*$ and restriction by $\Res_P:= p_*q^!$. For $I \subset I' \subset \Delta$, put $\Ind^{I'}_{I}:=\Ind_{P^{I'}_I}, $ and $\Res^{I'}_{I}:=\Res_{P^{I'}_I}$.

\begin{proof}[Proof of Theorem \ref{generalizedspringer}]
	For the first statement, recall an object  in $\{\textup{cuspidal data} \}  $ is a pair $(L,F)$, where $L$ is a Levi subgroup of $G$, and $F$ is a (irreducible perverse) cuspidal sheaf on $\cN_L$. The equivalence relation is given by $(L,F) \sim (L',F')$ if there is $g \in G,$ such that  $Ad_g(L)=L'$ and $Ad^*_g(F') \simeq F$. Denote by $C_I$ the isomorphism class of cuspidal sheaves on $\cN_{L_I}$. We need to show that the map $l:\coprod_{I \subset {\Delta}} C_I \to \{\textup{cuspidal data} \}/\sim$, defined via $(I, F \in C_I) \mapsto (L_I,F)$, is bijective. $l$ is surjective because any Levi subgroup of $G$ is conjugate to one of the $L_I$. To show injectivity, assume $(I,F \in C_I) \sim(I', F' \in C_{I'})$, i.e, there is $g \in G,$ such that $Ad_g(L_I)=L_{I'}$ and $Ad^*_g(F')=F$. Then $Ad_g(P^\Delta_I)$ and $P^\Delta_{I'}$ are two parabolic subgroup of $G$ containing $L_{I'}$ as Levi, then by above Lemma, they are conjugate. Hence $P^\Delta_I$ and $P^\Delta_{I'}$ are also conjugate, but both of them contains $B=P^\Delta_\emptyset$, so $P^\Delta_I = P^\Delta_{I'}$, and $I=I'$.  We have $\Ind_{P^\Delta_I}(F) \simeq \Ind_{P^\Delta_I}(F')$, so $F \otimes \CC[W^I] \simeq \Res_{P^\Delta_I}\Ind_{P^\Delta_I}(F) \simeq \Res_{P^\Delta_I}\Ind_{P^\Delta_I}(F') \simeq  F' \otimes \CC[W^I]$ (see e.g \cite{RR} Lemma 4.5), hence $F=F'$ in $C_I$.
	
	For the second statement, the argument follows very closely \cite{RR2}. For $F \in C_I$, Denote $Spr_{I,F}:=\Ind^\Delta_I F$ the corresponding generalized Springer sheaf, and $Spr := \oplus_{\{I \subset \Delta, F \in C_I\}} Spr_{I,F}$. By \cite{Lu7}, every irreducible perverse sheaf on $\N_G/G$ appear as direct summand of $Spr$. Since every object in $\Sh_c(\N/G)$ can be obtained from cones of maps between irreducible perverse sheaves, we conclude that $Spr$ Karoubi generates $\Sh_c(\N_G/G)$. We have:
    \begin{align*}
\Sh_c(\N/G)
&  \simeq \End(Spr)\textup{-perf} 
\qquad \text{(by Proposition~\ref{Karoubigenerate})} \\
& 	\simeq \bigoplus_{I,F} \End(Spr_{I,F})\textup{-perf}  \qquad (\Hom(Spr_{I,F},Spr_{I',F'})=0, \forall (I,F) \neq (I',F')  \text{ by \cite[Theorem 3.5]{RR}}) \\
& \simeq \bigoplus_{I,F} H^*(\End(Spr_{I,F}))\textup{-perf}  \qquad (\text{by Proposition~\ref{formalspringer}})
\\
&  \simeq \bigoplus_{I,F} \CC[W^I] \# \SS_{\frz^*_{I}[-2]}\textup{-perf}  \qquad (\text{by \cite[Proposition 4.6]{RR2}})
\\
&  \simeq \bigoplus_{I,F} \CC[W^I] \# \SS_{\frz_{I}[1]}\textup{-mod}_\textup{fd}  \qquad (\text{by Koszul duality, Corollary \ref{Koszul duality}})
\\
&  \simeq \bigoplus_{I \subset \Delta} (\CC[W^I] \# \SS_{\frz_{I}[1]}\textup{-mod}_\textup{fd})^{\oplus c_I}.
\end{align*}

Proof of the third statement is postponed to later this section.

\end{proof}	

We record the following Proposition for later use:
\begin{prop}
	\label{resind}
Let $J \subset J' \subset \Delta$. Under the identification in Theorem~\ref{generalizedspringer},
\begin{enumerate}
	\item 
the parabolic restriction $\Res^{J'}_J: \Sh_c(\N_{L_{J'}} /L_{J'}) \to \Sh_c(\N_{L_{J}}/L_J)$ is identified as 
$$\Res^{J'}_{J}: \bigoplus_{I' \subset J'} (\CC[W^{I'}_{J'}] \#   \SS_{\frz_{I'}[1]}  \textup{ -mod}_\textup{fd} )^{\oplus c_{I'}} \longrightarrow  \bigoplus_{I \subset J} (\CC[W^I_J] \#   \SS_{\frz_I[1]}  \textup{ -mod}_\textup{fd} )^{\oplus c_I}$$
where $\Res^{J'}_J$ is induced by the restriction along $W^{I'}_J \subset W^{I'}_{J'}$ for those summand labelled by $I' \subset J$, and is $0$ for summands $I' \not\subset J$. 
\item the parabolic induction $\Ind^{J'}_J: \Sh_c(\N_{L_{J}}/L_J) \to \Sh_c(\N_{L_{J'}}/L_{J'})$ is identified as 
$$\Ind^{J'}_{J}: \bigoplus_{I \subset J} (\CC[W^{I}_{J}] \#   \SS_{\frz_{I}[1]}  \textup{ -mod}_\textup{fd} )^{\oplus c_{I}} \longrightarrow  \bigoplus_{I' \subset J'} (\CC[W^{I'}_{J'}] \#   \SS_{\frz_{I'}[1]}  \textup{ -mod}_\textup{fd} )^{\oplus c_{I'}}$$
where $\Ind^{J'}_J$ is induced by the induction along $W^I_J \subset W^I_{J'}$.
\end{enumerate}
\end{prop}	
\begin{proof}
	(2) follows from the transitivity of parabolic inductions: $\Ind^{J'}_{J}  \circ \Ind^{J}_{I}  \simeq \Ind^{J'}_I$. (1) follows from (2) by adjunction.
\end{proof}	

Let $\Sh^{cusp}_{(c)}(\N/G) \subset \Sh_{(c)}(\N/G)$ be the full subcategory consist of object $F$, such that $\oplus_{I \subsetneq \Delta}\Res^{\Delta}_I F =0$. And $\Sh^{eis}_c(\N/G) \subset \Sh_c(\N/G)$ be the full subcategory Karoubi generated by the Image of $\oplus_{I \subsetneq \Delta} \Ind^\Delta_I$. Under the identification in last section, $\Sh^{eis}_c(\N/G) = 
\bigoplus_{I \subsetneq \Delta} (\CC[W^I] \# \SS_{\frz_{I}[1]}\textup{-mod}_\textup{fd})^{\oplus c_I},$ and 
$\Sh^{cusp}_c(\N/G) =(\CC[W^\Delta] \# \SS_{\frz_{\Delta}[1]}\textup{-mod}_\textup{fd})^{\oplus c_\Delta}=\SS_{\frz_{\g}[1]}\textup{-mod}_\textup{fd}^{\oplus c_G}.$


\begin{prop}
We have $\Sh^{cusp}(\mathcal{N}/G) \simeq \SS_{\frz_{\g}[1]}\textup{-mod}^{\oplus c_G}$, which naturally contains above equivalence
$\Sh^{cusp}_c(\N/G) \simeq \SS_{\frz_{\g}[1]}\textup{-mod}_\textup{fd}^{\oplus c_G}$.
\end{prop}

\begin{proof} 
	For $\mathscr{C} \subset \mathscr{D}$ full subcategory and $F \in \mathscr{D}$, denote $<\mathscr{C},F>=0$ if $\Hom_{\mathscr{D}}(E,F)=0$, for any $E \in \mathscr{C}$.
	By Corollary~\ref{finite orbit stack}, we have $\Sh(\N/G)^c \subset \Sh_c(\N/G)$.  This induces $\Sh(\N/G)^c= \Sh(\N/G)^{c,cusp} \oplus \Sh(\N/G)^{c,eis}$, where $\Sh(\N/G)^{c,cusp/eis}:=\Sh(\N/G)^c \cap \Sh_c^{cusp/eis}(\N/G)$.  We claim $\Sh(\N/G)^{c,cusp}$ compactly generated $\Sh^{cusp}(\N/G).$ Indeed, let $F \in \Sh^{cusp}(\N/G)$, such that $<\Sh(\N/G)^{c,cusp},F>=0$. We have $<\Sh(\N/G)^{c,eis},F>=0$ by adjunction, since $\Sh(\N/G)^{c,eis}$ is contained in $\Sh_c^{eis}(\N/G)$, which is Karoubi generated by the image of parabolic inductions. Hence we have $<\Sh(\N/G)^c,F>=0$, this implies $F=0$. This proves $\Sh(\N/G)^{c,cusp}$ compactly generated $\Sh^{cusp}(\N/G).$ In particular, $\Sh(\N/G)^{c,cusp}$ generates $\Sh_c^{cusp}(\N/G)$. Assume in addition that $G$ is semisimple, then $\Sh_c^{cusp}(\N/G)\simeq \Vect_{\textup{fd}}^{\oplus c_G}$, therefore $\Sh(\N/G)^{c,cusp}$ must equal $\Sh_c^{cusp}(\N/G) \simeq \Vect_{\textup{fd}}^{\oplus c_G}$, hence $\Sh(\N/G) =\Vect^{\oplus c_G}$. For general reductive $G$, write $G= (Z^0_G \times G_{der})/H$, where the derived group $G_{der}=[G,G]$  is semisimple, and $H = Z^0_G \cap G_{der}$ is a finite abelian group. We have by descent $\Sh^{cusp}(\N/G) \simeq (\Sh^{cusp}(\N/G_{der}) \otimes \Loc(BZ^0_G))^{\Loc(BH)} \simeq \bigoplus_{F \in C_{G_{der}}} (\Vect_F \otimes \Loc(BZ^0_G))^{\Loc(BH)} \simeq \bigoplus_{F \in C_{G_{der}}, H \text{ acts trivially on }F} \Sh(BZ^0_G/BH) \simeq \bigoplus_{F \in C_G} \SS_{\frz_{\g}[1]}\textup{-mod}.  $ 
\end{proof}	


\begin{proof}[Proof of Theorem~\ref{generalizedspringer} continued] Let $E_I \in (\mathbb{S}_{\frz_\frl{[1]}} \textup{ -perf})^{\oplus c_I}  \subset \Sh_c^{cusp}(\N_{L_I}/L_I)$  corresponding to $\mathbb{S}_{\frz_I{[1]}}^{\oplus c_I}$. Since $\Ind^{\Delta}_I$ has a continous right adjoint $\Res^{\Delta}_I$ (which is left adjoint to $\Ind^{-\Delta}_{-I}$), hence by Propsition~\ref{preserve compact object}, we have $ \oplus_{I \subset \Delta} \Ind^{\Delta}_I (E_I) $ is a compact generator of $\Sh(\N/G)$. By Propsition~\ref{resind}, $ \Ind^{\Delta}_I (E_I) \simeq (\CC[W^I]\#\mathbb{S}_{\frz_I{[1]}})^{\oplus c_I} \in (\CC[W^I]\#\mathbb{S}_{\frz_I{[1]}} \textup{-mod}_\textup{fd})^{\oplus c_I}.$   Hence $\Sh(\N/G) \simeq \End(\oplus_{I \subset \Delta} \Ind^{\Delta}_I (E_I) ) \textup{-mod} \simeq  \bigoplus_{I \subset \Delta} (\CC[W^I]\#\mathbb{S}_{\frz_I{[1]}} \textup{-mod})^{\oplus c_I}$
	
\end{proof}

\section{Derived Springer theory on  groups}

As we have seen in the proof of Theorem~\ref{generalizedspringer}, the Springer embedding $ \CC[W] \# \SS_{\frt^*[-2]} \textup{-mod}_{\textup{fd}} \hookrightarrow \Sh_c(\cN_G/G)$ relies on a purity result of \cite{RR2}. In this section, we combine the Springer embedding, Theorem~\ref{glue} and vanishing of certain higher homotopy groups (Corollary~\ref{homotopypushout}) to deduce the corresponding embedding for $\Ch(G)$ (Theorem~\ref{derivedspringer}).

Denote by $\Delta^1= \{ 0 \to 1  \}$ the category with two objects and one non-identity morphism between them.

\begin{prop}
Let  $H: (\Delta^1)^n \to \textup{Grp}$ be a colimit diagram of groups,  such that every arrow in $(\Delta^1)^n$ goes to an injective group homomorphism, and every $2$-dimensional face goes to a Cartesian square. Then the following are equivalent: 
\begin{enumerate}
 \item The induced diagram for classifying spaces $BH: (\Delta^1)^n \to \textup{Top}$ is a (homotopy) colimit diagram.
 \item Denote by $* \in (\Delta^1)^n$ the final object. The induced diagram $(\Delta^1)^n \to D^-(\ZZ[H_*])$, via $I \mapsto \ZZ[H_*/H_I]$ is a colimit diagram in the bounded above dg category of $\ZZ[H_*]$-modules.
 \item The total complex of the diagram in \textup{(2)}
 $$0 \rightarrow \bigoplus_{|I|=0}\ZZ[H_*/H_I] \rightarrow  \bigoplus_{|I|=1}\ZZ[H_*/H_I] \rightarrow ... \rightarrow \bigoplus_{|I|=n-1}\ZZ[H_*/H_I] \rightarrow \ZZ \rightarrow 0$$ 
 is acyclic, where $|I|:=i_1+i_2+...+i_n$, for $I=(i_1 i_2 ... i_n), i_k \in \{0,1\}$.
 \end{enumerate}
\end{prop}
\begin{proof} 

 Equivalence between (2) and (3): $\colim_{I \in (\Delta^1)^n \backslash *} \ZZ[H_*/H_I]$ is computed by the complex:
  $$0 \rightarrow \bigoplus_{|I|=0}\ZZ[H_*/H_I] \rightarrow  \bigoplus_{|I|=1}\ZZ[H_*/H_I] \rightarrow ... \rightarrow \bigoplus_{|I|=n-1}\ZZ[H_*/H_I] \rightarrow 0.$$ 
 
 Equivalence between (1) and (2): The argument follows closely \cite[Lemma 4]{McD}. let $f:X:=\colim_{I \in (\Delta^1)^n \backslash *} BH_I \to BH_*$ be the induced map. For $I=(i_1i_2...i_n),I'=(i'_1i'_2...i'_n) \in (\Delta^1)^n$, denote by $I \cap I':=(\min\{i_1,i'_1\} \min\{i_2,i'_2\} ... \min\{i_n,i'_n\})$. Then $X=\cup_{I \in (\Delta^1)^n \backslash *} BH_I$, and $BH_I \cap BH_{I'}=BH_{I \cap I'}$. Hence by cube version of Van Kampen Theorem,  $f_*:\pi_1(X) \to \pi_1(BH_*)$ is an isomorphism. To show (1), it is equivalent to show that for any $M$ local system of abelian groups on $BH_*$, the induced map $f^*:R\Gamma(BH_*,M) \to R\Gamma(X, f^*M) $ is an isomorphism. Note that $H_I \to H_*$ is injective, hence $\ZZ[H_I] \to \ZZ[H_*]$ is flat. Then $f^*$ can be identified as:  $f^*:RHom_{\ZZ[H_*]}(\ZZ,M) \to \lim RHom_{\ZZ[H_I]}(\ZZ,M) \simeq \lim RHom_{\ZZ[H_*]}(\ZZ[H_*/H_I],M) \simeq RHom_{\ZZ[H_*]}( \colim \;\ZZ[H_*/H_I],M)$.
 Then by Yoneda Lemma,  $f^*$ is an isomorphism for any $\ZZ[H_*]$-module $M$ if and only if $\colim \;\ZZ[H_*/H_I] \to \ZZ$ is an isomorphism.
\end{proof}

Let $\WW$ be a reflection group with the set of walls $\frH$. For $I \subset \frH$, denote by $\WW_I \subset \WW$ the subgroup generated by reflections in $I$. 

\begin{cor}
	\label{homotopypushout}
	Let $\WW$ acting on $V$ be an reflection group which is infinite, essential and irreducible, let $S$ be the set of walls of an alcove. View $P(S)$ the power set of $S$ as a category via inclusion of subsets. Then the functor $B\WW_I: P(S) \rightarrow \textup{Top}$, where $I \mapsto B\WW_I$ is a (homotopy) colimit diagram.
\end{cor}
\begin{proof} Fix an isomorphism $S \simeq \{1,2,...,n\}$, for $n=|S|$, then we have $P(S) \simeq (\Delta^1)^n$. Since $\WW$ is can be generated by elements in $\WW_{I }$, for $ I \subset S$; and relations of $\WW$ comes from some $\WW_I$, we see that $I \mapsto \WW_I$ is colimit diagram of groups, so it is suffices check (3) above is satisfied. This is a standard argument (which we learnt from Roman Bezrukavnikov), we include a brief proof for reader's convenience. The hyperplanes in $\frH$ gives a simplicial decomposition of $V$. Let $C$ be an alcove, then the closure $\overline{C}$ is a fundermental domain of the $\WW$ action on $V$. The stablizer of the $(n-1-|I|)$-simplex $\overline{C} \cap \bigcap_{\HH \in I} \HH$ is $\WW_I$. Hence the set of $i$-simplex of $V$ can be identified with $\coprod_{|I|=n-i-1} \WW/\WW_I $, and the complex 
	 $$0 \rightarrow \bigoplus_{|I|=0}\ZZ[\WW/\WW_I] \rightarrow  \bigoplus_{|I|=1}\ZZ[\WW/\WW_I] \rightarrow ... \rightarrow \bigoplus_{|I|=n-1}\ZZ[\WW/\WW_I] \rightarrow 0.$$ 
	is the simplicial chain complex computing the homology of $V$, which is $\ZZ$ since $V$ is contractible.
\end{proof}
Assume now $G$ is simple and simply-connected. Denote by $\widetilde{W}$ the affine Weyl group. 

\begin{thm}
	\label{derivedspringer}
	 There are fully-faithful embeddings of dg categories as direct summand
	$$\xymatrix{\CC[\widetilde{W}] \# \SS_{\frt^*[-2]} \textup{ -mod}_{\textup{fd}} 
	\ar@{^{(}->}[r]^-{\oplus} &  \Ch_c(G)	} $$
$$\xymatrix{\CC[\widetilde{W}] \# \SS_{\frt[1]} \textup{ -mod}
	\ar@{^{(}->}[r]^-{\oplus} &  \Ch(G)} $$
\end{thm}

\begin{proof}
We show the second embeding, and the first one follows from the same argument. By Theorem~\ref{generalizedspringer} (proved in next section), we have $\Sh(\N_{L_I}/L_I) \simeq  (\CC[W_I] \# \SS_{\frt[1]}  \textup{-mod})  \oplus \Sh(\N_{L_I}/L_I)^{\textup{np}}$, where  $\Sh(\N_{L_I}/L_I)^{\textup{np}}$ are those blocks whose Levi subgroup of the cusipdal data is not a maximal torus. Note that this direct sum decomposition is stable under parabolic restriction $\Sh(\N_{L_I}/L_I) \to \Sh(\N_{L_I'}/L_I')$. Passing to the limit, by Theorem~\ref{glue}, we have $\Ch(G)= \lim_{\{I \subsetneq \widetilde{\Delta} \}^{op}} \Sh(\N_{L_I}/L_I) = \lim (\CC[W_I] \# \SS_{\frt[1]}  \textup{-mod})  \oplus \lim \Sh(\N_{L_I}/L_I)^{\textup{np}}$.


Now in the notation of Section~\ref{group action on category}, take $A = \SS_{\frt[1]}$, and $\Gamma= \widetilde{W}$ acting on $A$ via the quotient $\widetilde{W} \to W$. 
We have $ \CC[\widetilde{W}] \# A\textup{-mod} 
\simeq  \lim_{B\widetilde{W}} Act \simeq \colim_{{B\widetilde{W}}^{op}} Act^L \simeq \colim_{I \subsetneq \widetilde{\Delta}} \colim_{B\widetilde{W}_I^{op}} Act^L 
\simeq \lim_{\{I \subsetneq \widetilde{\Delta}\}^{op}} \lim_{B \widetilde{W}_I} Act
\\ \simeq \lim_{\{I \subsetneq \widetilde{\Delta}\}^{op}}  (\CC[\widetilde{W}_I] \# \SS_{\frt[1]}  \textup{-mod}) 
\simeq \lim_{\{I \subsetneq \widetilde{\Delta}\}^{op}}  (\CC[{W}_I] \# \SS_{\frt[1]}  \textup{-mod}),$ where the third equivalence is by Corollary~\ref{homotopypushout}.

\end{proof}	

\section{Proof of the main theorem and its corollaries}



We identifiy $\widetilde{\Delta}$ with the set of walls of the fundamental alcove of $\widetilde{W}$ acting on $\AA=X_*(T) \otimes \RR$. For any $I \subset I' \subsetneq  \widetilde{\Delta}$,  denote $\widetilde{W}^I_{I'}:= N_{\widetilde{W}_{I'}}({\widetilde{W}_{I}})/{\widetilde{W}_{I}}$, and ${W}^I_{I'}:= N_{{W}_{I'}}({{W}_{I}})/{{W}_{I}}$, then $\pi$ induces $\widetilde{W}^I_{I'} \simeq {W}^I_{I'}.$ Denote by $\AA_I:= \cap_{\alpha \in I} \alpha^\perp$ the real affine subspace of $\AA_{\emptyset} := \AA$. Denote $V_I$ the vector space of translation of $\AA_I$. 

\begin{lem} 
	\label{weylgroup}
Assume $c_I >0$, then $\widetilde{W}^I$ acting on $\AA_I$ satisfies the assumption of Corollary~\ref{homotopypushout} with $S$ identifed with $\widetilde{\Delta}\backslash I$. For any $ I \subset I' \subsetneq \widetilde{\Delta}$, we have $(\widetilde{W}^I)_{I' \backslash I} = \widetilde{W}^I_{I'} $. 
\end{lem}
\begin{proof} 
Denote $A \subset \AA$ the fundamental alcove. Let $\Gamma_I:=$ the subgroup of $\widetilde{W}$ generated by $N_{\widetilde{W}_I'} \widetilde{W_I}$ for all $I'$ such that $I \subset I' \subsetneq \widetilde{\Delta}$. By \cite[Theorem 9.2]{Lu7}, we see that $\Gamma_I/\widetilde{W}_I$ is a reflection group on $\AA_I$ with alcove $A_I$, where $A_I$ is defined by  $\overline{A_I}=\overline{A} \cap \AA_I$; and simple reflections in $\Gamma_I/\widetilde{W}_I$ can be identified as $\widetilde{\Delta}\backslash I$, such that $(\Gamma_I/\widetilde{W}_I)_{I' \backslash I}= \widetilde{W}^I_{I'}$. The composition of maps $\overline{A} \cap \AA_I = \AA_I /\Gamma_I \to \AA_I/ N_{\widetilde{W}} \widetilde{W}_I \to \AA/ \widetilde{W} = \overline{A}$ is injective. Hence $\Gamma_I/\widetilde{W}_I = N_{\widetilde{W}} \widetilde{W}_I/\widetilde{W}_I=\widetilde{W}^I.$ 

\end{proof}	

\begin{proof}[Proof of Theorem \ref{main}]
For $J \subsetneq \widetilde{\Delta}$, we have $J = \Delta(L_J,T,P^J_\emptyset)$, hence by Theorem~\ref{generalizedspringer}, we have $Sh(\cN_{L_J}/L_J) \simeq \bigoplus_{I \subset J} (\CC[W^I_J] \#   \SS_{\frz_I[1]}  \textup{ -mod} )^{\oplus c_I}.$ 
Under this identification, for $J \subset J'$, the parabolic restriction 
\begin{equation}
\label{res}
\Res^{J'}_{J}: \bigoplus_{I' \subset J'} (\CC[W^{I'}_{J'}] \#   \SS_{\frz_{I'}[1]}  \textup{ -mod} )^{\oplus c_{I'}} \to  \bigoplus_{I \subset J} (\CC[W^I_J] \#   \SS_{\frz_I[1]}  \textup{ -mod} )^{\oplus c_I}
\end{equation}
 is identified as restriction along $W^{I'}_J \subset W^{I'}_{J'}$ for those $I' \subset J$ and $0$ for those $I' \not\subset J$. Now we compute

    \begin{align*}
\Ch(G)
&  \simeq \lim_{J \subsetneq \widetilde{\Delta}} \bigoplus_{I \subset J} (\CC[W^I_J] \#   \SS_{\frz_I[1]}  \textup{ -mod} )^{\oplus c_I}   
\qquad \text{(by Theorem~\ref{glue})} \\
& 	\simeq \bigoplus_{I \subsetneq \widetilde{\Delta}} (\lim_{I \subset J \subsetneq \widetilde{\Delta}} \CC[W^I_J] \#   \SS_{\frz_I[1]}  \textup{ -mod} )^{\oplus c_I}  \qquad \text{(by (\ref{res}))} \\
& \simeq \bigoplus_{I \subsetneq \widetilde{\Delta},c_I > 0} (\lim_{I \subset J \subsetneq \widetilde{\Delta}} \CC[\widetilde{W}^I_J] \#   \SS_{\frz_I[1]}  \textup{ -mod} )^{\oplus c_I}   
 \\
&  \simeq \bigoplus_{I \subsetneq \widetilde{\Delta}, c_I > 0} (\CC[ \widetilde{W}^I] \#   \SS_{\frz_I[1]}  \textup{ -mod} )^{\oplus c_I}  \qquad (\text{by Lemma~\ref{weylgroup}, applying the proof of Theorem~\ref{derivedspringer} to } \widetilde{W}^I).  
\end{align*}
\end{proof}

\begin{proof}[Proof of Corollary~\ref{spectral}]
	We have $X_*(Z_I^0) \subset \Lambda_I$ of finite index, this induces $\frs^*_I:= Lie(\check{S}_I) \simeq \frz_I^*$. 
We have $\QCoh(\mathcal{L}\check{S}_I / W^I) \simeq \QCoh ((\check{S}_I \times \frs^*_I[-1])/W^I) \simeq \CC[W^I]\#( \mathcal{O}(\check{S}_I) \otimes\mathcal{O}(\frz^*_I[-1])) \textup{ -mod} \simeq \CC[W^I]\#(\CC[\Lambda_I] \otimes \SS_{\frz_I[1]})\textup{ -mod} \simeq \CC[\widetilde{W}]\#\SS_{\frz_I[1]}\textup{ -mod}. $ 
\end{proof}

\begin{proof}[Proof of Corollary~\ref{irreducible character sheaves}]
The irreducible object in $\Coh^{\heartsuit}_0( \widehat{G}^{cl})$ is given by $\delta_{I,F,s,\rho} :=i_{I,F,s,*} \rho$, where $i_{I,F,s}: BW^I_s \hookrightarrow \widehat{G}^{cl}$ the closed embedding determined by $(I,F,s)$, and $\rho$ is viewed as an object in $\Coh(BW^I_s)$. For the second statement, $\Hom^*_{\Coh_0(\widehat{G})}(\delta_{I,F,s,\rho},\delta_{I',F',s',\rho'})=0$ for $(I,F,s) \neq (I',F',s')$; and $\Hom^*_{\Coh_0(\widehat{G})}(\delta_{I,F,s,\rho},\delta_{I,F,s,\rho'})=\Hom^*_{\Coh(BW^I_s)} (\rho, i^!_{I,F,s}\delta_{I,F,s,\rho'} ) =\Hom^*_{W^I_s} (\rho,\SS_{\frz_I^*[-1] \oplus \frz_I^*[-2]} \otimes \rho' )$.
\end{proof}	

\section{Acknowledgements}
We would like to thank Pramod Achar, Roman Bezrukavnikov, Joana Cirici, Dragos Fratila, Sam Gunningham, Quoc Ho, Geoffroy Horel, Daniel Juteau, Nitu Kitchloo, George Lusztig, David Nadler, Bertrand Toën and Zhiwei Yun for  helpful discussions. 
The author is grateful for the support of Prof. Tamas Hausel and the Advanced Grant “Arithmetic and Physics of Higgs moduli spaces” No. 320593 of the European Research Council.

\appendix
\section{Formality of generalized Springer sheaves}

In this appendix, we shall use a general ``purity implies formality'' theorem of Cirici-Hovey \cite{CH17} to duduce the formality of genenralized Springer sheaves.
 
\begin{defn}
A dg algebra $A$ is \textit{formal} if there is an isomorphism of dg algebras $A \simeq H^*(A)$.
\end{defn}
We 	denote $\Vect $  the symmetric monoidal dg category of vector spaces.
\begin{defn}
 Let $\mathscr{C}$ be a monoidal dg category. A lax/oplax monoidal functor $F:\mathscr{C} \to \Vect$ is \textit{formal} if there is an isomorphism of lax/oplax monoidal functors $F \simeq H^* \circ F$.
\end{defn}	

A dg algebra $A$ is equivalent to a monoidal functor $F_A: \textup{Disk}_1 \to \Vect^\otimes$, where $\textup{Disk}_1$ is the 1-Disk category. And $A$ is formal if and only if $F_A$ is formal.

\begin{defn} A  bounded (\textit{resp.} bounded below/above) \textit{mix Hodge complex} (MHC for short) is given by a filtered bounded (\textit{resp.} bounded below/above) cochain complex
	$(K_\QQ, W_\QQ)$ over $\QQ$, a bifiltered cochain complex $(K, W, F)$ over $\CC$, together with an isomorphism in the dg category of filtered complexes of $\CC$-vector spaces: $\varphi:(K_\QQ \otimes \CC,W_\QQ \otimes \CC) \simeq (K, W)$, such that:
	\begin{enumerate}
		\item The cohomology $H^i(K_\QQ)$ is finite dimensional, for all $i \in \ZZ$.
		\item The differential of $Gr^p_W K$ is strictly compatible with $F$.
		\item The filtration on $H^n(Gr^p_W K)$ induced by $F$ makes $H^n(Gr^p_W K_\QQ)$ into a pure Hodge structure of weight $p + n$.
    \end{enumerate}
A mixed Hodge complex is \textit{pure} if $H^n(Gr^p_W K)=0$, for $p \neq 0$.
\end{defn}

The bounded (\textit{resp.} bounded below/above) mixed Hodge complexes forms a symmetric monoidal dg category. Denote by $\MHC^{b}$ (\textit{resp.} $\MHC^{+}/\MHC^{-}$). The pure mixed Hodge complexes forms a fully subcategory $\MHC_{pure}^{?} \subset \MHC^{?}$, $? =b,+$ or $-$.

\begin{thm} 
	\label{formality}
The forgetful functor $U:\MHC_{pure}^{?} \to  \Vect,$ via $((K_\QQ,W_\QQ),(K,W,F),\varphi) \mapsto K_\QQ$ is formal. 
\end{thm}

\begin{proof}
	For $?=b$, this is proved in \cite{CH17}. The truncation functor $\tau^{\leq n}: \MHC^{+} \to \MHC^b$ is oplax monoidal. Denote $U^{\leq n}:= \tau^{\leq n} \circ U$. We have $H^* \circ U^{\leq n} \simeq U^{\leq n}$, therefore $H^* \circ U \simeq H^* \circ \varprojlim U^{\leq n} \simeq \varprojlim  H^* \circ U^{\leq n} \simeq \varprojlim U^{\leq n} \simeq U $ as oplax monoidal functors, where the second equivalence follows from the fact that the maps $H^* \circ U^{\leq n+1} \to H^* \circ U^{\leq n}$ are surjective, and therefore the inverse limit $\varprojlim  H^* \circ U^{\leq n}$ in $\Vect$ agree with the inverse limit as graded vector spaces. The argument for $\MHC^{-}$ is similar.
\end{proof}

A sheaf is of \textit{geometric origin} is it can be produced from constant sheaves via  $*/!$-pushforward,$*/!$-pullback along morphisms of algebraic varieties, external tensor product and taking direct summand. \cite{CH17} construct a MHC structure on the (co)chains $C^*(Z)$ on an algebraic variety $Z$.  This gives a MHC structure on  the sections $R\Gamma(X,F)$, for any $F$ on $X$ of geometric origin. We thank Sam Gunningham for sketching a proof of the following formality result for $L=T$:

\begin{prop}
	\label{formalspringer}
 Let $F$ be a cuspidal sheaf on $\calN_L/L$. Then the endomorphism dg algebra $\End(\Ind_P(F))$ is formal. 
\end{prop}

\begin{proof}
Denote $i:  Z:= \calN_P/P \times_{\N_G/G} \calN_P/P \to \N_G/G $ the map induced by $p$. By \cite[(8.6.4) and Propsition 8.6.35]{CG09}, we have an isomorphism of dg algebras $\End(\Ind_P(F)) \simeq R\Gamma(Z,i^!((q^*F)^\vee \boxtimes q^*F)) \simeq R\Gamma(Z,i^!(q^*(F^\vee) \boxtimes q^*F))$.  Since $F$ and $F^\vee$ is of geometric origin \cite{Lu1}, $R\Gamma(Z,i^!(q^*(F^\vee) \boxtimes q^*F))$ carries a MHC structure, which is compatible with its algebra structure (since it is also contructed geometically).  The Frobenius action on $H^*(\End(\Ind_P(F)))$ is pure \cite[Proposition 3.5]{RR2}, and so is the action on $H^*(Z,i^!((q^*F)^\vee \boxtimes q^*F))$. We conclude that the $R\Gamma(Z,i^!(q^*(F^\vee) \boxtimes q^*F))$ lifts to an algebra object in $\MHC_{pure}^{+}$, and therefore is formal by Theorem~\ref{formality}.

\end{proof}

\bibliographystyle{alpha}
\bibliography{paper}

\end{document}